\documentclass[a4paper]{amsart}
\usepackage{aliascnt, amssymb, enumerate, enumitem, hyperref, mathtools, bbm}
\usepackage[all]{xy}

\def\today{\number\day\space\ifcase\month\or   January\or February\or
   March\or April\or May\or June\or   July\or August\or September\or
   October\or November\or December\fi\   \number\year}

\theoremstyle{definition}
\newtheorem{lma}{Lemma}[section]

\newaliascnt{thmCt}{lma}

\aliascntresetthe{thmCt}

\newtheorem*{thm*}{Theorem}

\newaliascnt{corCt}{lma}
\newtheorem{cor}[corCt]{Corollary}
\aliascntresetthe{corCt}

\newaliascnt{propCt}{lma}
\newtheorem{prop}[propCt]{Proposition}
\aliascntresetthe{propCt}

\newaliascnt{pgrCt}{lma}

\aliascntresetthe{pgrCt}

\newaliascnt{dfCt}{lma}
\newtheorem{df}[dfCt]{Definition}
\aliascntresetthe{dfCt}

\newaliascnt{remCt}{lma}
\newtheorem{rem}[remCt]{Remark}
\aliascntresetthe{remCt}

\newaliascnt{remsCt}{lma}

\aliascntresetthe{remsCt}

\newaliascnt{egCt}{lma}
\newtheorem{eg}[egCt]{Example}
\aliascntresetthe{egCt}

\newaliascnt{egsCt}{lma}

\aliascntresetthe{egsCt}

\newaliascnt{qstCt}{lma}

\aliascntresetthe{qstCt}

\newaliascnt{pbmCt}{lma}

\aliascntresetthe{pbmCt}

\newaliascnt{notaCt}{lma}

\aliascntresetthe{notaCt}

\newaliascnt{cnjCt}{lma}

\aliascntresetthe{cnjCt}

\newcounter{theoremintro}
\newtheorem{thmintro}[theoremintro]{Proposition}

\newcommand{\ph}{\varphi}

\newcommand{\Z}{{\mathbb{Z}}}

\newcommand{\N}{{\mathbb{N}}}

\newcommand{\ind}{{\mathbbm{1}}}

\newcommand{\ev}{{\mathrm{ev}}}

\newcommand{\andeqn}{\,\,\, {\mbox{and}} \,\,\,}

\newcommand{\ca}{$C^*$-algebra}
\newcommand{\cas}{$C^*$-algebras}

\newcommand{\ov}{\overline}

\title{Duality of partial Rokhlin dimension}

\date{\today}
\thanks{The author was supported by Kungl. Vetenskapsakademiens stiftelser.
}

\author{Jan Gundelach}
\address[Jan Gundelach]{Department of Mathematical Sciences, Chalmers University of
Technology and University of Gothenburg, Gothenburg SE-412 96, Sweden.}
\email[]{jangund@chalmers.se}
\urladdr{www.chalmers.se/en/persons/jangund/}

\begin{document}

\begin{abstract}
    We extend the notion of representability dimension to partial actions and introduce a notion of dual representability dimension for global actions by finite abelian groups. We show that the Rokhlin dimension of a partial action by a finite abelian group agrees with the dual representability dimension of the dual action on the partial crossed product, while the representability dimension of a partial action agrees with the Rokhlin dimension of its dual.
\end{abstract}
\maketitle

\section{Introduction}
Finite group actions on \cas\, with the Rokhlin property have a long history that dates back to the late 1970s. Even before this notion has been established in its modern form, Kishimoto studied instances of Rokhlin actions in \cite{Kis_fpUHF_1978}. The dynamical system of an action with the Rokhlin property is particularly well behaved. For example, the Rokhlin property allows to pass on many properties from the $C^*$-algebra both to the resulting fixed point algebra and to the crossed product. Among others, this includes being simple, a Kirchberg algebra, of real rank zero, or having finite stable rank or nuclear dimension. Furthermore, the Rokhlin property passes to tensor product actions and quotients, and facilitates K-theoretical and Cuntz semigroup computations. This is why the Rokhlin property is used early on for generalization attempts of the classification oriented work of Kirchberg and Phillips to an equivariant framework.\par Among many other results, Izumi characterizes in \cite[Lemma 3.8]{Izu_finiteI_2004} when a global action by a finite abelian group has the Rokhlin property in terms of the dual action. The dual notion is approximate representability allowing to check certain examples of action with the Rokhlin property more easily. The theory of Rokhlin actions ramified in many directions since then. Among others, one branch of generalizations is from finite to compact groups, from separable to not necessarily separable \cas, from the Rokhlin property to Rokhlin dimensions, and from global to partial actions. Following this branch, Izumi's result generalizes accordingly to compact groups in \cite[Theorem VI.4.2]{Gar_thesis_2015} and \cite[Theorem 4.27]{BarSza_sequentially_2016}, and to Rokhlin and representability dimensions in \cite[Theorem 1.14]{GarHirSan_rokhlin_2017}. In this paper, we continue this theme for partial actions in the sense that the Rokhlin dimension of a partial action agrees with the dual representability dimension of its dual. Using terminology from Landstad duality, the dual representability dimension is the representability dimension in \cite[Definition 1.10]{GarHirSan_rokhlin_2017} with an additional compatibility relation for duals of partial actions that is automatic for duals of global actions. 
\begin{thmintro}\label{prop intro dual rep}
    Let $G$ be a finite abelian group, and let $\alpha = ((A_g)_{g\in G},(\alpha_g)_{g\in G})$ be a partial action  on a separable \ca\, $A$. Then $\dim_{\textup{Rok}}(\alpha) = \dim_{\widehat{\textup{rep}}}(\widehat{\alpha})$.
\end{thmintro}
Similarly, we generalize the representability dimension to partial actions leading to a complementary characterization: The representability dimension of a partial action by a finite abelian group agrees with the global Rokhlin dimension of the dual action.

\begin{thmintro}\label{prop intro dual rok}
    Let $G$ be a finite abelian group, and let $\alpha = ((A_g)_{g\in G},(\alpha_g)_{g\in G})$ be a partial action on a separable \ca\, $A$. Then $\dim_{\textup{Rok}}(\widehat{\alpha}) = \dim_{\textup{rep}}(\alpha)$.
\end{thmintro}

\section{Partial Rokhlin and representability dimensions}\label{partial Rok section}
In this section, we recall the definition of Rokhlin dimension for partial actions and introduce the representability dimension for partial actions. We begin by recalling the definitions of a partial action and its fixed point algebra; see also \cite[Section 6]{Exe_mono_2017} and \cite[Definition 4.1]{AbaGefGar_decomposable_2021}.

\begin{df}\label{partial action}
    Let $A$ be a \ca\, and let $G$ be a group. A \textit{partial action} is a collection $\alpha = ((A_g)_{g\in G},(\alpha_g)_{g\in G})$ of ideals $A_g \lhd A$ and isomorphisms $\alpha_g\colon A_{g^{-1}}\to A_g$ for all $g\in G$ with $\alpha_e = \textup{id}_A$ and such that, for all $g,\,h\in G$, the map $\alpha_{gh}$ extends $\alpha_g\circ \alpha_h$ whenever the composition is well-defined. The \textit{fixed point algebra} of $\alpha$ is given by \[A^\alpha := \{ a\in A\colon \alpha_g(a_{g^{-1}}a) = \alpha_g(a_{g^{-1}})a\,\, \textup{for all}\,\, g\in G\,\, \textup{and all}\,\,a_{g^{-1}}\in A_{g^{-1}}\}.\] We say that a partial action is \textit{global} if $A_g=A$ for all $g\in G$.
\end{df}

\begin{eg}\label{restricted to I example}
    For any global action $\alpha\colon G \curvearrowright A$ and any ideal $I\lhd A$, restricting to the dynamics inside $I$, we obtain a partial action $\alpha\vert_I$ by setting $I_g:= I\cap \alpha_g(I)$ and $(\alpha\vert_I)_g := \alpha_g\vert_{I_{g^{-1}}}$ for $g\in G$. Partial actions of this form are called \textit{globalizable} since they admit a global enveloping action on $A$.
\end{eg}
We recall the definition of Rokhlin dimension from \cite[Definition 2.1]{AbaGefGar_parRok_2022}.

\begin{df}\label{part Rok unital epsilon version}
    Let $\alpha = ((A_g)_{g\in G},(\alpha_g)_{g\in G})$ be a partial action of a finite group $G$ on a \ca\, A. For $d\in \N$, we say that $\alpha$ has \textit{Rokhlin dimension at most d} if for every $\varepsilon >0$ and finite subset $F\subseteq A$, there exist tuples $(f_g^{(j)})_{g\in G}\in \prod_{g\in G}A_g$ of positive contractions called \textit{Rokhlin towers} for $(F,\varepsilon)$ for $j=0,\cdots,d$, satisfying:
    \begin{enumerate}[label=(\roman*)]
        \item $\|f_g^{(j)}f_h^{(j)}a\|<\varepsilon$ for all $g,h\in G$ with $g\neq h$, $a\in F$, and $j=0,\cdots, d$;
        \item $\|(f_g^{(j)}b -bf_g^{(j)})a\|< \varepsilon$ for all $g\in G$, $a,b\in F$, and $j=0,\cdots, d$;
        \item $\|(\sum_{j=0}^d \sum_{g\in G}f_g^{(j)})a -a\| < \varepsilon$ for all $a\in F$;
        \item $\|(\alpha_g(f_h^{(j)}x) - f_{gh}^{(j)}\alpha_g(x))a\| < \varepsilon$ for all $g,h\in G$, $a\in F$, $x\in A_{g^{-1}}\cap F$, and $j=0,\cdots,d$.
    \end{enumerate}
    If existent, the minimal such $d$ is called the \textit{Rokhlin dimension} of $\alpha$ and is denoted by $\dim_{\textup{Rok}}(\alpha)$. If $\dim_{\textup{Rok}}(\alpha)= 0$, we say that $\alpha$ has the \textit{Rokhlin property}.
\end{df}
We aim for a reformulation of \autoref{part Rok unital epsilon version} that is more convenient for our purposes. Recall from \cite{WinZac_completely_2009} that a completely positive contractive map $\ph\colon C\to A$ between \cas\, is said to have \textit{order zero} if $\ph(x)\ph(y)=0$ whenever $x,y\in C_+$ satisfy $xy=0$. To rephrase the conditions on Rokhlin towers in \autoref{part Rok unital epsilon version} in terms of completely positive contractive order zero maps, we introduce the central sequence algebra.

\begin{df}\label{seq alg}
    Let $A$ be a \ca, let $l^\infty(\N,A)$ be the algebra of norm-bounded $A$-valued sequences with entry-wise operations, and let $c_0(\N,A)$ be the ideal of sequences that converge to zero in norm. We define the \textit{sequence algebra} as $A_\infty := l^\infty(\N,A)/c_0(\N,A)$. We identify $A$ with the equivalence classes of constant sequences and define the \textit{central sequence algebra} as its relative commutant $A_\infty \cap A'$. We further define the \textit{multiplier algebra} of $A$ as $M(A):= \{b\in A^{**}\colon bA\cup Ab\subseteq A\}$ and say that $M(A)$ is the \textit{idealizer} of $A$ in $A^{**}$.\par Let $\alpha = ((A_g)_{g\in G},(\alpha_g)_{g\in G})$ be a partial action by a group $G$. Since $A$ is an ideal in $M(A)$, we can regard the partial action $\alpha$ as a partial action on $M(A)$ as well. Applying $\alpha$ entry-wise, in turn, induces a partial action on $A_\infty$ (or $M(A)_\infty$ respectively) that we denote by \[\alpha^\infty:=(((A_g)_\infty)_{g\in G},(\alpha^\infty_{g})_{g\in G}).\] 
\end{df}

\begin{rem}\label{fix rem}
    Since \cas\, admit approximate identities, the fixed point algebra for a global action $\alpha\colon G\curvearrowright A$ of a group $G$ on a \ca\, $A$ simplifies to
    \[A^\alpha := \{ a\in A\colon \alpha_g(a) = a\,\, \textup{for all}\,\,g\in G\}.\] For a global action $\alpha$ of a finite group, we can further identify $(A_\infty)^{\alpha^\infty}$ and $(A^\alpha)_\infty$ using the observation that, for $x\in (A_\infty)^{\alpha^\infty}$, any representing sequence $(x_n)_n$ satisfies \[\frac{1}{|G|}\sum_{h\in G} \alpha_h(x_n) -x_n \to 0\] as $n\to \infty$. Thus, without loss of generality, $x$ can be represented by an $A^\alpha$-valued sequence, namely $(\frac{1}{|G|}\sum_{h\in G} \alpha_h(x_n))_n$.
\end{rem}
The following proposition is the analogue of \cite[Definition 1.3]{GarHirSan_rokhlin_2017} for partial actions.

\begin{prop}\label{part Rok cen seq version}
    Let $\alpha = ((A_g)_{g\in G},(\alpha_g)_{g\in G})$ be a partial action of a finite group $G$ on a separable \ca\, $A$. Let $\texttt{Lt}\colon G\curvearrowright C(G)$ be the canonical action by left translation. For $d\in \N$, the inequality $\dim_{\textup{Rok}}(\alpha)\leq d$ is equivalent to the existence of completely positive contractive order zero maps $\ph^{(j)}\colon C(G) \to M(A)_\infty\cap A'$, for $j=0,\cdots,d$, that satisfy
    \begin{enumerate}
        \item $\ph^{(j)}(\delta_g)A_\infty \cup A_\infty \ph^{(j)}(\delta_g)\subseteq (A_g)_\infty$ for all $g\in G$, and $j=0,\cdots,d$;
        \item $\sum_{j=0}^d \ph^{(j)}(\ind_G) = 1$;
        \item $\alpha^\infty_g(\ph^{(j)}(f)x) = \ph^{(j)}(\texttt{Lt}_g(f))\alpha_g(x)$ for all $f\in C(G),\, g\in G,\,x\in A_{g^{-1}}$, and $j=0,\cdots,d$.
    \end{enumerate}
\end{prop}

\begin{proof}
    If $\dim_{\textup{Rok}}(\alpha)\leq d$, then by separability we may choose an increasing sequence of finite sets $F_n\subseteq A$ with dense union in $A$ and Rokhlin towers $f_{g,n}^{(j)}\in A_g\subseteq M(A)$ for $(F_n, 2^{-n})$. Since these elements are positive contractions, the induced map $\ph^{(j)}\colon C(G)\to M(A)_\infty$ given by $\ph^{(j)}(\delta_g):= [(f_{g,n}^{(j)})_{n\in \N}]$ for $g\in G$ is a completely positive contractive map satisfying $(1)$. Now, using strict convergence in $M(A)$, we have that (i) in \autoref{part Rok unital epsilon version} implies that $\ph^{(j)}$ is an order zero map, (ii) that it takes values in $A'$, (iii) implies $(2)$, and finally (iv) implies $(3)$ by linearity, density and continuity.\par Conversely, given completely positive contractive order zero maps $\ph^{(j)}\colon C(G) \to A_\infty\cap A'$ as above, the Choi-Effros lifting theorem in \cite{ChoEff_lifting_1976} allows to find completely positive contractive lifts $\Phi^{(j)}\colon C(G) \to l^\infty(\N, M(A))$ such that, for all $g\in G$, the sequence $(f_{g,n}^{(j)})_{n\in \N}:=\Phi^{(j)}(\delta_g)$ can be arranged to consist of positive contractions in $A_g$ because of $(1)$. For given $(F,\varepsilon)$, using the order zero property, the commutant image, and the properties $(2)$ and $(3)$ of $\ph^{(j)}$ in the sequence algebra, eventually there is an $n\in \N$ such that these $f_{g,n}^{(j)}$ satisfy the corresponding strict convergence properties (i) to (iv) in \autoref{part Rok unital epsilon version}. This yields Rokhlin towers for $(F,\varepsilon)$.
\end{proof}

\begin{rem}\label{idealizer rem}
    Given that $b\in (A^{**})_\infty$ is in $M(A)_\infty$ if and only if $bA_\infty\cup A_\infty b\subseteq A_\infty$, condition $(1)$ in \autoref{part Rok cen seq version} reads as a relative idealizer version that for all $g\in G$ and $j=0,\cdots,d$, the set $\ph^{(j)}(\delta_g)A_\infty\cup A_\infty \ph^{(j)}(\delta_g)$ is not only contained in $A_\infty$, but in $(A_g)_\infty$ already. In particular, $(1)$ is automatic for global actions.
\end{rem}
By introducing a similar idealizer condition as in the passage from global to partial Rokhlin dimension, we continue to generalize the complementary dimension notion of representability dimension from global actions, as defined in \cite[Definition 1.10]{GarHirSan_rokhlin_2017}, to partial actions.

\begin{df}\label{ar dim eps version}
    Let $\beta = ((B_g)_{g\in G},(\beta_g)_{g\in G})$ be a partial action of a finite group $G$ on a \ca\, $B$. For $d\in \N$, we say that $\beta$ is \textit{approximately representable with d colors} if for every $\varepsilon >0$ and finite subset $F\subseteq B$, there exist contractions $x_g^{(j)}\in B_{g^{-1}}$ for $g\in G$ and $j=0,\cdots,d$ with $x_e^{(j)}\in B_+$ satisfying:
    \begin{enumerate}[label=(\roman*)]
        \item $\|((x_g^{(j)})^*x_g^{(j)} - x_g^{(j)}(x_g^{(j)})^*)b\|<\varepsilon$ for all $b\in F$;
        \item $\|(x_g^{(j)}x_h^{(j)} -x_{e}^{(j)}x_{gh}^{(j)})b\|< \varepsilon$ for all $g,h\in G$, and $b\in F$;
        \item $\|(\beta_g(ax_h^{(j)}) - \beta_g(a)x_{ghg^{-1}}^{(j)})b\| < \varepsilon$ for all $g,h\in G$, $a\in B_{g^{-1}}\cap F$, and $b\in F$;
        \item $\|(\sum_{j=0}^d x_e^{(j)})b-b\| < \varepsilon$ for all $b\in F$;
        \item $\|x_g^{(j)}b - \beta_g(b)x_g^{(j)}\| < \varepsilon$ for all $b\in B_{g^{-1}}\cap F$.
    \end{enumerate}
    If existent, the minimal such $d$ is called the \textit{representability dimension} of $\beta$ and is denoted by $\dim_{\textup{rep}}(\beta)$. If $\dim_{\textup{rep}}(\beta)= 0$, we say that $\beta$ is \textit{approximately representable}.
\end{df}
Once again, we aim for a reformulation of the representability dimension that is more convenient for our purposes. To generalize unitary representations analogously to passing from unital $*$-homomorphisms to completely positive contractive order zero maps as in \cite[Theorem 3.3]{WinZac_completely_2009}, we introduce the notion of order zero representations; see \cite[Definition 1.7]{GarHirSan_rokhlin_2017}.

\begin{df}\label{order zero repr}
    Let $B$ be a \ca\, and let $G$ be a group. A map $v\colon G\to B$ is an \textit{order zero representation} if
    \begin{enumerate}
        \item $v_g$ is a normal contraction for all $g\in G$, and $v_e\in B_+$;
        \item $v_gv_h = v_ev_{gh}$ for all $g,h\in G$;
        \item $v_g^*=v_{g^{-1}}$ for all $g\in G$.
    \end{enumerate}
\end{df}
The following is \cite[Proposition 1.8]{GarHirSan_rokhlin_2017} and describes order zero representations as dilated unitary representations.

\begin{prop}\label{order zero repr map prop}
    Let $B$ be a \ca, let $G$ be a group, and let $v\colon G\to B$ be an order zero representation. Then there exist a projection $p\in B^{**}$ and a unitary representation $\pi\colon G\to \mathcal{U}(pB^{**}p)$ commuting with $v_e$, such that $v_g = v_e\pi_g$ for all $g\in G$. In particular, order zero representations of $G$ are in one-to-one correspondence with completely positive contractive order zero maps from $C^*(G)$. 
\end{prop}
This allows us to rephrase the representability dimension in terms of the order zero representations analogously to the global case as in \cite[Theorem 1.12 (2)]{GarHirSan_rokhlin_2017} with $v^{(j)} = \rho_j^2$ for all $j=0,\cdots,d$.
 
\begin{prop}\label{ar cen seq version}
    Let $B$ be a separable \ca, let $G$ be a group, and let $\beta = ((B_g)_{g\in G},(\beta_g)_{g\in G})$ be a partial action. For $d\in \N$, the inequality $\dim_{\textup{rep}}(\beta)\leq d$ is equivalent to the existence of completely positive contractive order zero maps $v^{(j)}\colon C^*(G) \to M(B)_\infty$, for $j=0,\cdots,d$, that satisfy
     \begin{enumerate}
        \item $v_g^{(j)}B_\infty \cup B_\infty v_g^{(j)}\subseteq (B_{g^{-1}})_\infty$ for all $g\in G$, and $j=0,\cdots,d$;
        \item $\beta^\infty_g(bv^{(j)}_h) = \beta^\infty_g(b)v_{ghg^{-1}}^{(j)}$ for all $j=0,\cdots,d$, $g,h\in G$, and $b\in (B_{g^{-1}})_\infty$;
        \item $\sum_{j=0}^d v_e^{(j)}=1$;
        \item $v_g^{(j)}b = \beta_g(b)v_g^{(j)}$ for all $j=0,\cdots,d$ and $g\in G,\,b\in B_{g^{-1}}\subseteq B_\infty$.
    \end{enumerate}
\end{prop}

\begin{proof}
    If $\dim_{\textup{rep}}(\beta)\leq d$, then by separability we may choose an increasing sequence of finite sets $F_n\subseteq B$ with dense union in $B$ and $x_{g,n}^{(j)}\in B_{g^{-1}}\subseteq M(B)$ for $(F_n, 2^{-n})$. Conditions (i) and (ii) ensure that the induced map $v^{(j)}\colon G\to M(B)_\infty$ given by $v^{(j)}_g:= [(x_{g,n}^{(j)})_{n\in \N}]$ for $g\in G$ is an order zero representation. It integrates to a completely positive contractive order zero map $v^{(j)}\colon C^*(G) \to M(B)_\infty$ by \autoref{order zero repr map prop} satisfying $(1)$ because of our choices for $x_{g,n}^{(j)}\in B_{g^{-1}}\subseteq M(B)$. Using strict convergence in $M(B)$, we have that (iii) in \autoref{ar dim eps version} implies $(2)$, that (iv) implies $(3)$, and finally that (v) implies $(4)$.\par Conversely, given completely positive contractive order zero maps $v^{(j)}\colon C^*(G) \to M(B)_\infty$ as above, $(2)$ combined with the Choi-Effros lifting theorem in \cite{ChoEff_lifting_1976} allows us to find completely positive contractive lifts $V^{(j)}\colon C^*(G) \to l^\infty(\N, M(B))$ such that, for all $g\in G$ and $n\in \N$, the entries $(V^{(j)}_g)_n$ are contractions in $B_{g^{-1}}$ with $(V^{(j)}_e)_n\in B_+$. For given $(F,\varepsilon)$, using the definition of an order zero representation as well as properties $(2)$ to $(4)$ of $v^{(j)}$ in the sequence algebra, there is $n\in \N$ such that the $(V^{(j)}_g)_n$ satisfy the properties (i) to (v) in \autoref{part Rok unital epsilon version}.
\end{proof}

\begin{rem}\label{dimrep zero global}
    Note that the first condition $(1)$ in \autoref{ar cen seq version} ensures that an approximately representable partial action $\beta = ((B_g)_{g\in G},(\beta_g)_{g\in G})$ is necessarily global. Indeed, the single completely positive contractive order zero map $v^{(0)}\colon C^*(G) \to M(B)_\infty$ is unital by $(3)$ and thus a $*$-homomorphism. It restricts to a unitary representation $v\colon G \to M(B)_\infty$ such that $v_gA_\infty \cup A_\infty v_g\subseteq (A_{g^{-1}})_\infty$ for all $g\in G$. Hence $A_\infty= v_gv_g^*A_\infty\subseteq v_g(A_{g})_\infty\subseteq (A_{g^{-1}})_\infty$. In particular, $A_g= A$ for all $g\in G$.
\end{rem}
We recall the concept of a partial representation; see \cite[Definition 2.2]{DokExePic_parrepr_2000}.

\begin{df}
    Let $G$ be a group. A \textit{partial representation} of $G$ on a unital \ca\, $B$ is a map $\pi\colon G\to B$ such that, for all $g,h\in G$, we have
    \begin{enumerate}
        \item $\pi(e) = 1_B$,
        \item $\pi(g^{-1}) = \pi(g)^*$,
        \item $\pi(g)\pi(h)\pi(h^{-1}) = \pi(gh)\pi(h^{-1})$.
    \end{enumerate}
\end{df}
Partial actions allow for the construction of full and reduced partial crossed products as completions of the section algebra for an associated Fell bundle over the group; see \cite[Section 11]{Exe_mono_2017}. Since we are dealing with finite groups, by \cite[Theorem 4.7]{Exe_amenFell_1997} both constructions coincide and can be described more easily in terms of the following algebra.

\begin{df}\label{crossed product def}
    Let $\alpha = ((A_g)_{g\in G},(\alpha_g)_{g\in G})$ be a partial action of a finite group $G$ on a \ca\, $A$. The associated \textit{partial crossed product} is the set $A \rtimes_\alpha G$ of formal linear combinations of $a_gu_g$, for $g\in G$ and $a_g\in A_g$, with convolution and involution given by
    \[(a_gu_g)(b_hu_h) :
        = \alpha_g(\alpha_{g^{-1}}(a_g)b_h)u_{gh} \andeqn (a_gu_g)^* := \alpha_{g^{-1}}(a_g^*)u_{g^{-1}}\]
    for all $g,h\in G,\,a_g\in A_g,\,a_h\in A_h$. A \textit{covariant representation} on a unital \ca\, $B$ is a pair $(\rho,\pi)$ consisting of a $*$-homomorphism $\rho\colon A\to B$ and a partial representation $\pi\colon G\to B$ such that $\rho(\alpha_g(x)) = \pi_g\rho(x)\pi_{g^{-1}}$ for all $g\in G$ and $x\in A_{g^{-1}}$. We turn the partial crossed product into a $C^*$-algebra by equipping it with the $C^*$-norm \[\bigg\|\sum_{g\in G}a_gu_g \bigg\| := \sup \bigg\{\bigg\|\sum_{g\in G}\rho(a_g)\pi_g \bigg\|\colon (\rho,\pi) \,\,\textup{is a covariant representation} \bigg\}.\]
\end{df}
Associativity of the partial crossed product is established in \cite[Proposition 2.4]{Exe_asspartcp_1997}.

\begin{lma}\label{u is partial implementing rep}
    Let $\alpha = ((A_g)_{g\in G},(\alpha_g)_{g\in G})$ be a partial action of a finite group $G$ on a \ca\, $A$. Then $u\colon G\to (A\rtimes_\alpha G)^{**}$ is a partial representation such that $(A_gu_e)u_g = u_g(A_{g^{-1}}u_e)$ and $u_gu_g^*$ is the projection onto $(A_gu_e)^{**}$. Furthermore, $u$ implements the partial action in the sense that $\alpha_g(x) = u_gxu_{g^{-1}}$ for all $g\in G$ and $x\in A_{g^{-1}}$.
\end{lma}

\begin{proof}
    \cite[Lemma 6.2]{DokExe_asspartcp_2005} shows that $u\colon G\to (A\rtimes_\alpha G)^{**}$ is a partial representation such that $(A_gu_e)u_g = u_g(A_{g^{-1}}u_e)$ and $u_gu_g^*$ is the projection onto $(A_gu_e)^{**}$. Finally, for given $g\in G$, let $(e_i^{(g)})_i$ be a bounded approximate identity for the ideal $A_g$. Using the convolution identity in \autoref{crossed product def}, we compute $(e_i^{(g)}u_g)(xu_{g^{-1}}) = e_i^{(g)}\alpha_g(x)u_e$ for all $x\in A_{g^{-1}}$. Since $(e_i^{(g)}u_g)_i$ converges strictly to $u_g$, the bounded approximate identity can be omitted and this shows the claim.
\end{proof}
We conclude this section by recalling that partial crossed products are topologically graded algebras in the sense of \cite[Definition 3.4]{Exe_amenFell_1997}.

\begin{df}\label{top grading}
    Let $G$ be a finite group, let $B$ be a \ca, and let $B_g\subseteq B$ be closed linear subspaces for all $g\in G$. We say that the collection $(B_g)_{g\in G}$ is a \textit{topological grading} for $B$ if there exists a bounded linear map $E\colon B \to B_e$ that is the identity on $B_e$ and vanishes on $B_g$ for each $g\neq e$ and such that, for all $g,h\in G$, we have
    \begin{enumerate}
        \item $B_gB_h\subseteq B_{gh}$,
        \item $B_g^* = B_{g^{-1}}$,
        \item $\ov{\textup{span}}(\bigcup_{g\in G}B_g) = B$.
    \end{enumerate}
\end{df}
The following proposition is an immediate consequence of \cite[Proposition 3.2]{Exe_amenFell_1997} and \cite[Theorem 3.3]{Exe_amenFell_1997}.

\begin{prop}\label{cp grading prop}
    Let $\alpha = ((A_g)_{g\in G},(\alpha_g)_{g\in G})$ be a partial action of a finite group $G$ on a \ca\, $A$. Then $(A_gu_g)_{g\in G}$ is a topological grading for $A\rtimes_\alpha G$.
\end{prop}

\section{Duality of partial Rokhlin dimension}\label{dual of pRok section}
In this section, we consider partial actions by finite abelian groups with finite Rokhlin dimension. We first import the terminology of Landstad duality and then introduce a dual approximate representability dimension for the dual action. Finally, we show that the dual approximate representability dimension of the dual action agrees with the Rokhlin dimension of the original action. We end this section with the dual characterization that the representability dimension as in \autoref{ar cen seq version} of a partial action by a finite abelian group agrees with the global Rokhlin dimension of the dual action as in \autoref{part Rok cen seq version}.\par
We begin by recalling the canonical dual action by the dual group on the partial crossed product. This is a special case of the theory of dual coactions.

\begin{df}\label{dual action}
    Let $G$ be abelian, and let $\alpha = ((A_g)_{g\in G},(\alpha_g)_{g\in G})$ be a partial action on a \ca\, $A$. The \textit{dual action} $\widehat{\alpha}\colon \widehat{G} \curvearrowright A \rtimes_\alpha G$ is given by
    \[ \widehat{\alpha}_\chi(a_gu_g) := \chi(g)a_gu_g\]
    for all $\chi\in \widehat{G},\,g\in G$ and $a_g\in A_g$.
\end{df}

\begin{prop}\label{fp of dual action}
    Let $G$ be abelian, and let $\alpha = ((A_g)_{g\in G},(\alpha_g)_{g\in G})$ be a partial action on a \ca\, $A$. Then $(A\rtimes_\alpha G)^{\widehat{\alpha}} = Au_e$.
\end{prop}

\begin{proof}
    The $\supseteq$ inclusion is trivial. Conversely, any $\sum_{g\in G} a_gu_g \in (A\rtimes_\alpha G)^{\widehat{\alpha}}$ satisfies $\sum_{g\in G}(\chi(g)-1)a_gu_g = 0$ for all $\chi\in \widehat{G}$. By \cite[Proposition 17.9]{Exe_mono_2017}, for reduced partial crossed products, we have $\sum_{g\in G}(\chi(g)-1)a_gu_g=0$ if and only if $(\chi(g)-1)a_g=0$ for all $g\in G$, this forces $a_g=0$ for $g\neq e$ and shows the claim.
\end{proof}

Note that the dual action $\widehat{\alpha}$ is global even if $\alpha$ is partial. In the following, we want to recall the characterization of when a global action by $\widehat{G}$ is the dual of a partial action by $G$, that is, we want to highlight those actions $\beta \colon \widehat{G} \curvearrowright B$ that are of the form $\widehat{\alpha}\colon \widehat{G} \curvearrowright A \rtimes_\alpha G$ for some partial action $\alpha \colon G \curvearrowright A$.

\begin{df}\label{spectral subspaces}
    Let $G$ be abelian, and let $\beta \colon \widehat{G} \curvearrowright B$ be a global action on a \ca\, $B$. Let $g\in G$ and define the \textit{g-spectral subspace} as 
    \[ B_g := \{b\in B\colon \beta_\chi(b) = \chi(g)b \,\,\, \textup{for all}\,\,\, \chi \in \widehat{G}\}\] and define the \textit{g-spectral range ideal} as $D_g:= \ov{\textup{span}}(\{bc^*\colon b,c\in B_g\})\lhd B_e$. Let $p_g\in B^{**}$ be the projection onto $D_g^{**}$, and extend the notion of multipliers on the fixed point algebra $B^\beta= B_e$ consistently to general spectral subspaces by $M(B_g) := \{b\in p_gB^{**}p_{g^{-1}}\colon D_gb \cup bD_{g^{-1}} \subseteq B_g\}$.
\end{df}
If $B= A\rtimes_\alpha G$ is a partial crossed product and if $\beta = \widehat{\alpha}$ is the dual action, then the spectral subspaces $B_g=A_gu_g$ form the topological grading in \autoref{cp grading prop}. In this case, the spectral range ideals of the fixed point algebra are $D_g= A_gu_e$ and by \autoref{u is partial implementing rep}, the implementing partial representation satisfies $u_g\in M(B_g)$ and $u_gu_g^* = p_g$ for all $g\in G$.\par The following proposition is a version of Landstad duality by Quigg and Raeburn and agrees with \cite[Theorem 4.1]{QuiRae_Landstad_1997} in the special case of abelian groups and dual actions instead of dual coactions.

\begin{prop}\label{Landstad duality}
    Let $G$ be a finite abelian group and let $\beta \colon \widehat{G} \curvearrowright B$ be a global action on a \ca\, $B$. For $g\in G$, let $p_g$ and $M(B_g)$ be as in \autoref{spectral subspaces}. Then the following are equivalent:
    \begin{enumerate}
        \item There are a partial action $\alpha = ((A_g)_{g\in G},(\alpha_g)_{g\in G})$ and an equivariant isomorphism $(\beta,B)\cong (\widehat{\alpha},A\rtimes_\alpha G)$.
        \item There is a partial representation $m\colon G\to B^{**}$ such that for all $g\in G$ we have $m_g\in M(B_g)$ and $m_gm_g^* = p_g$.
    \end{enumerate} 
\end{prop}
For a finite abelian group $G$, the dual group $\widehat{G}\subseteq C(G)$ generates the Kronecker delta functions via $\delta_g = \frac{1}{|G|}\sum_{\chi\in \widehat{G}}\overline{\chi(g)}\chi$ for all $g\in G$. This observation together with \autoref{ar cen seq version} and the existence of spectral range ideals for global actions by abelian groups motivate the following refinement of representability dimension.

\begin{df}\label{drepr dim}
    Let $G$ be a finite abelian group and let $\beta \colon \widehat{G} \curvearrowright B$ be a global action on a \ca\, $B$. For $d\in \N$, we say that $\beta$ is \textit{dually approximately representable with d colors}, if there are completely positive contractive order zero maps $v^{(j)}\colon C^*(\widehat{G}) \to M(B)_\infty^{\beta^\infty}$ for $j=0,\cdots,d$ that satisfy
     \begin{enumerate}
        \item $(\sum_{\chi\in \widehat{G}}\overline{\chi(g)}v_\chi^{(j)}) B^\beta_\infty \cup B^\beta_\infty (\sum_{\chi\in \widehat{G}}\overline{\chi(g)}v_\chi^{(j)})\subseteq (D_g)_\infty$ for all $j=0,\cdots,d$ and $g\in G$;
        \item $\sum_{j=0}^d v_{1_G}^{(j)}=1$;
        \item $v_\chi^{(j)}b = \beta_\chi(b)v_\chi^{(j)}$ for all $j=0,\cdots,d$ and $\chi\in \widehat{G},\,b\in B\subseteq B_\infty$.
    \end{enumerate}
    If existent, the minimal such $d$ is the \textit{dual representability dimension} of $\beta$ and denoted by $\dim_{\widehat{\textup{rep}}}(\beta)$. If $\dim_{\widehat{\textup{rep}}}(\beta)=0$, we say that $\beta$ is \textit{dually approximately representable}.
\end{df}
Note that $(1)$ in \autoref{drepr dim} is an additional condition compared to \autoref{ar cen seq version} for a global action $\beta \colon \widehat{G} \curvearrowright B$ showing that $\dim_{\textup{rep}}(\beta) \leq \dim_{\widehat{\textup{rep}}}(\beta)$. A dually approximately representable action $\beta$ is not automatically the dual of a partial action $\alpha\colon G\curvearrowright A$, but even if it is, the inequality can be strict. The following example can be seen as the dual of \cite[Example 3.2]{AbaGefGar_parRok_2022}.

\begin{eg}\label{repr dual repr example}
    Write ${\Z}/{2\Z} = \{0,1\}$ additively and $\widehat{{\Z}/{2\Z}} = \{\pm\}$ multiplicatively. Consider $A_0= C_0(0,1)\oplus C_0(0,1)$ and $A_1 = C_0((0,1)\setminus\{\frac{1}{2}\})\oplus C_0((0,1)\setminus\{\frac{1}{2}\})$. For any $f\in C([0,1]/\Z)$, define the shift operator \[\sigma(f)(t + \Z) := f\left(t+ \frac{1}{2} + \Z \right)\] for all $t\in [0,1]$. The involution $\alpha_1\colon A_1\to A_1$ given by $\alpha_1(\xi,\eta):= (\sigma(\eta),\sigma(\xi))$ for all $\xi,\eta\in C_0((0,1)\setminus\{\frac{1}{2}\})$ defines a partial action $\alpha = ((A_g)_{g\in \Z/2\Z},(\alpha_g)_{g\in \Z/2\Z})$.\par
    We claim that its dual action $\widehat{\alpha}\colon \{\pm\}\curvearrowright A_0u_0 \oplus A_1u_1$ satisfies $\dim_{\textup{rep}}(\widehat{\alpha}) < \dim_{\widehat{\textup{rep}}}(\widehat{\alpha})$. First, note that $\widehat{\alpha}$ is approximately representable with unitaries $\widetilde{v}_+= 1= (\ind_{(0,1)},\ind_{(0,1)})\in M(A_0)_\infty$ and $\widetilde{v}_{-} = (-\ind_{(0,1)},\ind_{(0,1)})\in M(A_0)_\infty$.\par To reach a contradiction, assume that $\widehat{\alpha}$ is dually approximately representable as well, that is, there exist self-inverse unitaries $v_+,v_{-}\in M(A_0u_0)_\infty$ that satisfy
    \begin{enumerate}
        \item $(v_+ - v_-)(A_0u_0)_\infty \subseteq (A_1u_0)_\infty$;
        \item $v_+=1= (\ind_{(0,1)},\ind_{(0,1)})$;
        \item $v_-(a_0u_0+a_1u_1) = (a_0u_0-a_1u_1)v_-$ for all $a_0\in A_0,\, a_1\in A_1$.
    \end{enumerate}
    Choose a sequence of functions $(\xi_n,\eta_n)\in C_b(0,1)\oplus C_b(0,1)$ that represent $v_-$. Since both coordinates of elements in $A_1$ vanish at $\frac{1}{2}$, condition $(1)$ implies that \[\lim_{n\to \infty}\xi_n \left(\frac{1}{2}\right) = \lim_{n\to \infty}\eta_n\left(\frac{1}{2}\right) = 1.\] \par At the same time, for all $a_1\in A_1$, condition $(3)$ states that \[v_-a_1u_1 = -(a_1u_1)v_- = -\alpha_1^\infty(\alpha_1(a_1)v_-)u_1.\] Writing $a_1=(f,g)$ for $f,g\in C_0((0,1)\setminus\{\frac{1}{2}\})$, this means that \[(\xi_nf,\eta_ng)+\alpha_1((\sigma(g)\xi_n,\sigma(f)\eta_n)) = (\xi_nf+\sigma(\sigma(f)\eta_n),\eta_ng+\sigma(\sigma(g)\xi_n))\in A_1\] converges uniformly to zero as $n\to \infty$. In particular, for all $t\in (0,1)\setminus \{\frac{1}{2}\}$, we have
    \[\lim_{n\to \infty}\xi_n \left(t\right) = -\lim_{n\to \infty}\eta_n\left(t\right)\]
    contradicting the limit in the previous paragraph.
\end{eg}
This shows that the dual representability dimension does not always agree with the representability dimension. However, if $\beta$ is the dual of a global action, then condition $(1)$ in \autoref{drepr dim} is implied by $v_\chi^{(j)} \in M(B)_\infty^{\beta^\infty}$ for all $j=0,\cdots,d$ and $\chi\in \widehat{G}$ since $D_g=B^\beta$ for all $g\in G$. Thus we have $\dim_{\textup{rep}}(\widehat{\alpha}) = \dim_{\widehat{\textup{rep}}}(\widehat{\alpha})$ for global actions $\alpha\colon G\curvearrowright A$.\par The following proposition is the first duality result of this section.

\begin{prop}\label{Rok dim is AR dim of dual}
    Let $G$ be a finite abelian group, and let $\alpha = ((A_g)_{g\in G},(\alpha_g)_{g\in G})$ be a partial action  on a separable \ca\, $A$. Then $\dim_{\textup{Rok}}(\alpha) = \dim_{\widehat{\textup{rep}}}(\widehat{\alpha})$.
\end{prop}

\begin{proof}
    We aim to show that $\dim_{\textup{Rok}}(\alpha)\leq d$ if and only if $\dim_{\widehat{\textup{rep}}}(\widehat{\alpha})\leq d$. Firstly, by \autoref{part Rok cen seq version}, $\dim_{\textup{Rok}}(\alpha)\leq d$ is equivalent to the existence of completely positive contractive order zero maps $\ph^{(j)}\colon C(G) \to M(A)_\infty\cap A'$ for $j=0,\cdots,d$ as in \autoref{part Rok cen seq version}. We express these maps in terms of order zero representations. Recall that sending a character $\chi\in \widehat{G}\subseteq C(G)$ to its standard unitary defines the canonical Fourier isomorphism $C(G)\cong C^*(\widehat{G})$. Using this isomorphism and \autoref{order zero repr map prop}, any completely positive contractive order zero map $\ph\colon C(G) \to M(A)_\infty\cap A'$ defines an order zero representation $v^\ph\colon \widehat{G}\to M(A)_\infty\cap A'$ given by $v^\ph_\chi := \ph(\chi)$ for all $\chi\in \widehat{G}\subseteq C(G)$. Conversely, any order zero representation $v\colon \widehat{G}\to M(A)_\infty\cap A'$ defines a completely positive contractive order zero map $\ph^v\colon C(G) \to M(A)_\infty\cap A'$ given by $\ph^v(\delta_g) = \frac{1}{|G|}\sum_{\chi\in \widehat{G}}\ov{\chi}(g)v_\chi$ for all $g\in G$. These two constructions are mutually inverse to each other.\par
    Secondly, by \autoref{fix rem} and \autoref{fp of dual action}, we may identify $M(A)_\infty$ and $M(A\rtimes_\alpha G)_\infty^{\widehat{\alpha}^\infty}$ as well as $A_g$ and $D_g$ for all $g\in G$. Thus we have that $\dim_{\widehat{\textup{rep}}}(\widehat{\alpha})\leq d$ is equivalent to the existence of order zero representations $v^{(j)}\colon \widehat{G}\to M(A)_\infty$, for $j=0,\cdots,d$, that satisfy $(1)$ and $(2)$ as in \autoref{part Rok cen seq version} and 
    \[ v_\chi^{(j)}u_eb = \widehat{\alpha}_\chi(b)v_\chi^{(j)}u_e 
    \]
    for all $\chi\in \widehat{G},\,b\in A\rtimes_\alpha G,\, j=0,\cdots,d$. Note that this final equation for $b\in Au_e$ corresponds to the statement that the order zero representations commute with constant sequences in $Au_e$, that is, $v^{(j)}_\chi \in M(A)_\infty\cap A'$ for all $\chi\in \widehat{G}$ and $j=0,\cdots,d$. \par To sum up, both dimension inequality statements are equivalent to the existence of order zero representations $v^{(j)}\colon \widehat{G}\to M(A)_\infty\cap A'$, for $j=0,\cdots,d$, satisfying $(1)$ and $(2)$ as in \autoref{part Rok cen seq version} and one dimension specific additional property that we focus on next. It suffices to show that the property
    \begin{equation}\label{v implement alphdual}
        v_\chi^{(j)}xu_{g^{-1}} = \widehat{\alpha}_\chi(xu_{g^{-1}})v_\chi^{(j)}
    \end{equation} for all $\chi\in \widehat{G},\,g\in G\setminus\{e\},\, x\in A_{g^{-1}}$, and $j=0,\cdots,d$ is equivalent to the final property in \autoref{part Rok cen seq version} for $f=\chi \in \widehat{G}\subseteq C(G)$ and $\texttt{Lt}_g(\chi)= \ov{\chi(g)}\chi$. Spelled out, this second property reads as
    \begin{equation}\label{v alpha equivariance}
        \alpha^\infty_g(v^{(j)}_\chi x) = \ov{\chi}(g)\alpha_g(x)v_\chi^{(j)}
    \end{equation}
    for all $\chi\in \widehat{G},\,g\in G\setminus\{e\},\, x\in A_{g^{-1}}$, and $j=0,\cdots,d$. To show that the properties \eqref{v implement alphdual} and \eqref{v alpha equivariance} are equivalent, we fix $\chi\in \widehat{G},\, g\in G\setminus \{e\},\,x\in A_{g^{-1}}$, and $j\in \{0,\cdots,d\}$. By \autoref{u is partial implementing rep}, left multiplication \[u_g(-)\colon (A_{g^{-1}})_\infty u_{g^{-1}}\to (A_g)_\infty\] is a bijection. We claim that left multiplication by $u_g$ turns \eqref{v implement alphdual} into \eqref{v alpha equivariance}. Indeed, using \eqref{v implement alphdual} in the second step, we observe
    \[ \alpha^\infty_g(v^{(j)}_\chi x) = u_g(v_\chi^{(j)}x)u_{g^{-1}} = u_g\widehat{\alpha}_\chi(xu_{g^{-1}})v_\chi^{(j)} = \ov{\chi}(g)u_gxu_{g^{-1}}v_\chi^{(j)} = \ov{\chi}(g)\alpha_g(x)v_\chi^{(j)}.
    \] This finishes the proof.
\end{proof}
We end this section with a complementary duality result. We show that the representability dimension as in \autoref{ar cen seq version} of a partial action by a finite abelian group agrees with the global Rokhlin dimension of the dual action as in \autoref{part Rok cen seq version}.

\begin{prop}\label{dual rok and par repr dim}
    Let $G$ be a finite abelian group, and let $\alpha = ((A_g)_{g\in G},(\alpha_g)_{g\in G})$ be a partial action on a separable \ca\, $A$. Then $\dim_{\textup{Rok}}(\widehat{\alpha}) = \dim_{\textup{rep}}(\alpha)$.
\end{prop}

\begin{proof}
    We aim to show that $\dim_{\textup{Rok}}(\widehat{\alpha})\leq d$ if and only if $\dim_{\textup{rep}}(\alpha)\leq d$ by first rephrasing $\dim_{\textup{Rok}}(\widehat{\alpha})\leq d$ in terms of order zero representations \[w^{(j)}\colon G\to M(A\rtimes_\alpha G)_\infty\cap (A\rtimes_\alpha G)'\] for $j=0,\cdots,d$, and then verifying that those representations in question are precisely of the form $w^{(j)}_g = v_{g^{-1}}^{(j)}u_g$, for $g\in G$ and $j=0,\cdots,d$, with order zero representations $v^{(j)}$ as in the reformulation of $\dim_{\textup{rep}}(\alpha)\leq d$ in \autoref{ar cen seq version}.\par Since the dual action is global, by \autoref{idealizer rem} and \autoref{part Rok cen seq version}, the inequality $\dim_{\textup{Rok}}(\widehat{\alpha})\leq d$ is equivalent to the existence of $(\texttt{Lt},\widehat{\alpha}^\infty)$-equivariant completely positive contractive order zero maps $\ph^{(j)}\colon C(\widehat{G}) \to M(A\rtimes_\alpha G)_\infty\cap (A\rtimes_\alpha G)'$ for $j=0,\cdots,d$ such that $\sum_{j=0}^d \ph^{(j)}(\ind_{\widehat{G}}) = 1$. Note that the maps $\ph^{(j)}$ are determined by their values on the generating set $\widehat{\widehat{G}} \subseteq C(\widehat{G})$. The group isomorphism $G\cong \widehat{\widehat{G}}$ given by $g\mapsto \ov{\textup{ev}}_g$ determines the concrete Fourier isomorphism $C^*(G)\cong C(\widehat{G})$ given by $u_g\mapsto \ov{\textup{ev}}_g$. Since the images of the canonical unitaries under a completely positive contractive order zero map form an order zero representation by \autoref{order zero repr map prop}, the maps $\ph^{(j)}$ can equivalently be described in terms of order zero representations $w^{(j)}\colon G\to M(A\rtimes_\alpha G)_\infty\cap (A\rtimes_\alpha G)'$. Concretely, we have that $w^{(j)}_g := \ph^{(j)}(\ov{\ev}_g)$ for all $j=0,\cdots,d$ and $g\in G$ as well as $\ph^{(j)}(\delta_\chi) = \frac{1}{|G|}\sum_{g\in G}{\chi}(g)w_g^{(j)}$ for all $j=0,\cdots,d$ and $\chi\in \widehat{G}$. In particular, $w^{(j)}_e = \ph^{(j)}(\ind_{\widehat{G}})$. We would like to express the Rokhlin dimension purely in terms of these order zero representations now. Since $\texttt{Lt}_\tau(\ov{\ev}_g(\chi)) = \ov{\ev}_g(\ov{\tau}\chi) = \tau(g)\ov{\chi(g)}$, the equivariance condition is equivalent to $\widehat{\alpha}_\chi^\infty(w_g^{(j)}) = \chi(g)w_g^{(j)}$ for all $\chi\in \widehat{G},\,g\in G$, and $j=0,\cdots,d$.\par Since $(A_gu_g)_{g\in G}$ is a topological grading for $A\rtimes_\alpha G$ by \autoref{cp grading prop}, the multiplier algebra inherits a corresponding topological grading $(M(A_gu_g))_{g\in G}$ consisting of spectral subspace multiplier algebras as in \autoref{spectral subspaces}. Fix $j\in \{0,\cdots,d\}$, $g\in G$, and representing sequences $b_h =(b_{h,n})_{n\in \N} \in \ell^\infty(\N,M(A_hu_h))$ such that $w_{g}^{(j)} = \sum_{h\in G}(b_h + c_0(\N,M(A_hu_h)))$. Then $\widehat{\alpha}^\infty_\chi(\sum_{h\in G}(b_h + c_0(\N,M(A_hu_h)))) = \chi(g)\sum_{h\in G}(b_h + c_0(\N,M(A_hu_h)))$ implies that $(\chi(h)-\chi(g))b_h \in c_0(\N,M(A_hu_h))$ for all $\chi \in \widehat{G}$ and $h\in G$. In particular, we have $b_h \in c_0(\N,M(A_hu_h))$ for all $h\in G\setminus\{g\}$ and thus $w_{g}^{(j)} = b_{g}+ c_0(\N, M(A\rtimes_\alpha G))$. Conversely, any order zero representation with $w^{(j)}_g\in M(A_gu_g)_\infty$ for all $g\in G$ clearly satisfies the desired equivariance condition. To sum up, $\dim_{\textup{Rok}}(\widehat{\alpha})\leq d$ is equivalent to the existence of order zero representations $w^{(j)}\colon G\to M(A\rtimes_\alpha G)_\infty\cap (A\rtimes_\alpha G)'$ for $j=0,\cdots,d$ such that $w^{(j)}_g\in M(A_gu_g)_\infty$ and $\sum_{j=0}^d w_e^{(j)}= 1.$ This concludes the first reformulation.\par For the second reformulation, let $\alpha$ be approximately representable with $d$ colors and order zero representations $v^{(j)}\colon G\to M(A^\alpha)_\infty$ for $j=0,\cdots,d$. We claim that $w^{(j)}_g := v^{(j)}_{g^{-1}}u_g$ for $g\in G$ and $j=0,\cdots,d$ define order zero representations $w^{(j)}\colon G\to M(A\rtimes_\alpha G)_\infty\cap (A\rtimes_\alpha G)'$ as in the first reformulation of Rokhlin dimension in the previous paragraph. Furthermore, the relation $v^{(j)}_{g^{-1}} = w^{(j)}_{g}u_{g^{-1}}$ for $g\in G$ and $j=0,\cdots,d$ allows to reconstruct the maps $v^{(j)}$ from the maps $w^{(j)}$.\par We begin by showing that, under the claimed correspondence, maps $v^{(j)}\colon G\to M(A^\alpha)_\infty$ with $v_g^{(j)}A_\infty \cup A_\infty v_g^{(j)}\subseteq (A_{g^{-1}})_\infty$ for all $j=0,\cdots,d$ and $g\in G$ are order zero representations if and only if the corresponding maps $w^{(j)}\colon G\to M(A\rtimes_\alpha G)_\infty$ with $w^{(j)}_g\in M(A_gu_g)_\infty$ for all $g\in G$ are order zero representations. Indeed, for fixed $g\in G$ and $j\in\{0,\cdots,d\}$, by definition of the fixed point algebra in \autoref{partial action} and \autoref{fix rem}, $v_g^{(j)}\in M(A^\alpha)_\infty$ means that for all $h\in G$ and $a_{h}\in (A_{h})_\infty$, we have
    \[ \alpha_{h^{-1}}^\infty(a_hv_g^{(j)}) = \alpha_{h^{-1}}^\infty(a_h)v_g^{(j)}.\] In particular, the assumption $v_g^{(j)}A_\infty \subseteq (A_{g^{-1}})_\infty$ allows us to deduce the identity $\alpha_{g^{-1}}^\infty(v^{(j)}_g) = v_g^{(j)}p_g$. Hence $(v^{(j)}_g)^* = v^{(j)}_{g^{-1}}$ holds if and only if 
    \[(w^{(j)}_g)^* = (v^{(j)}_{g^{-1}}u_g)^*= \alpha_{g^{-1}}^\infty((v^{(j)}_{g^{-1}})^*)u_{g^{-1}} = \alpha_{g^{-1}}^\infty(v^{(j)}_g)u_{g^{-1}} = v^{(j)}_gu_{g^{-1}} = w^{(j)}_{g^{-1}}\] holds. Further note that, under this assumption on the adjoint, the conditions that $v^{(j)}_{g^{-1}}u_g$ should be a normal contraction and that $v_e^{(j)}u_e$ should be positive are equivalent to $v^{(j)}_g$ being a normal contraction and $v^{(j)}_e$ being positive. Finally, for fixed $g,h\in G$ and $j\in\{0,\cdots,d\}$, we compute
    \[\alpha_g^\infty(\alpha_{g^{-1}}^\infty(v_{g^{-1}}^{(j)})v_{h^{-1}}^{(j)}) = \alpha_g^\infty(\alpha_{g^{-1}}^\infty(v_{g^{-1}}^{(j)}))v_{h^{-1}}^{(j)} = v_{g^{-1}}^{(j)}v_{h^{-1}}^{(j)}.\] 
    Thus, the condition $v_{g^{-1}}^{(j)}v_{h^{-1}}^{(j)} =v_e^{(j)}v_{(gh)^{-1}}^{(j)}$ is equivalent to
    \[ w^{(j)}_gw^{(j)}_h = \alpha_g^\infty(\alpha_{g^{-1}}^\infty(v_{g^{-1}}^{(j)})v_{h^{-1}}^{(j)})u_{gh} = v_{g^{-1}}^{(j)}v_{h^{-1}}^{(j)}u_{gh}= v_e^{(j)}v_{(gh)^{-1}}^{(j)}u_{gh} = w^{(j)}_ew^{(j)}_{gh}.\]
    This shows the claim that the maps $w^{(j)}$ are order zero representations if and only if the maps $v^{(j)}$ are.\par It remains to relate the implementation condition on $v_g^{(j)}\in M(A^\alpha)_\infty$ to the relative commutant condition on $w_g^{(j)}\in M(A\rtimes_\alpha G)_\infty\cap (A\rtimes_\alpha G)'$. Assume that conjugation by $v_g^{(j)}\in M(A)_\infty^{\alpha^\infty}$ implements $\alpha_g$ for all $g\in G$ and $j\in\{0,\cdots,d\}$. By using the implementation assumption in the fifth step and the fixed point assumption in the sixth step, we compute for all $g,h\in G$ and $a_hu_h\in A_hu_h\subseteq (A\rtimes_\alpha G)_\infty$ that \begin{align*}
        w^{(j)}_{g}a_hu_h &= \alpha_g^\infty(\alpha_{g^{-1}}^\infty(v^{(j)}_{g^{-1}})a_h)u_{gh} \\ 
        &= \alpha_g^\infty(\alpha_{g^{-1}}^\infty(v^{(j)}_{g^{-1}})a_h p_g)u_{gh} \\
         &= \alpha_g^\infty(\alpha_{g^{-1}}^\infty(v^{(j)}_{g^{-1}}))\alpha_g^\infty(a_h p_g)u_{gh} \\
        &= v^{(j)}_{g^{-1}}\alpha_g(a_hp_g)u_{gh} \\
        &= a_hp_gv^{(j)}_{g^{-1}}u_{hg} \\
        &= \alpha_h^\infty(\alpha_{h^{-1}}(a_h)v^{(j)}_{g^{-1}})u_{hg} \\
        &= a_hu_hw^{(j)}_{g}.
    \end{align*} Conversely, for fixed $g\in G$ and $j\in\{0,\cdots,d\}$, a direct inspection of the multiplication law in \autoref{crossed product def} shows that $w_g^{(j)}$ commutes with $\alpha_{g^{-1}}(a)u_e$ for $a\in A_g$ if and only if $\alpha_g^\infty(\alpha_{g^{-1}}^\infty(v_{g^{-1}})\alpha_{g^{-1}}(a))u_g = \alpha_{g^{-1}}(a)v_{g^{-1}}^{(j)}u_g$. By \autoref{u is partial implementing rep}, right multiplication \[(-)u_g\colon (A_g)_\infty \to (A_{g}u_g)_\infty\] is a bijection. Thus, the commutation relation is equivalent the implementation relation $v_{g^{-1}}^{(j)}a = \alpha_{g^{-1}}(a)v_{g^{-1}}^{(j)}$. Likewise, for fixed $g,h\in G$ and $j\in\{0,\cdots,d\}$, we have that if $w_{g^{-1}}^{(j)}$ commutes with $u_{h^{-1}}$, then 
    \begin{align*}
        \alpha_h^\infty(xv^{(j)}_g) &= \alpha_h^\infty(xw^{(j)}_{g^{-1}}u_g)\\
        &= u_hx(w^{(j)}_{g^{-1}}u_g)u_{h^{-1}}\\
        &= u_hxu_{h^{-1}}(w^{(j)}_{g^{-1}}u_g)\\
        &= \alpha_h^\infty(x)w^{(j)}_{g^{-1}}u_g\\
        &= \alpha_h^\infty(x)v^{(j)}_g
    \end{align*} for all $x\in (A_{h^{-1}})_\infty$. Thus $v^{(j)}_{g}\in M(A)_\infty^{\alpha^\infty}$. This concludes the second reformulation of approximate representability with $d$ colors. The featuring order zero representations $v^{(j)}\colon G\to M(A^\alpha)_\infty$ for $j=0,\cdots,d$ are of the desired form if and only if $w^{(j)}_g = v^{(j)}_{g^{-1}}u_g$ for $g\in G$ and $j=0,\cdots,d$ define order zero representations $w^{(j)}\colon G\to M(A\rtimes_\alpha G)_\infty\cap (A\rtimes_\alpha G)'$ as in the first reformulation of Rokhlin dimension at most $d$.\par Combining both reformulations shows $\dim_{\textup{rep}}(\alpha)\leq d$ if and only if $\dim_{\textup{Rok}}(\widehat{\alpha})\leq d$, as claimed. 
\end{proof}

\begin{cor}\label{dual Rok forces predual global AR}
    Let $G$ be a finite abelian group, and let $\alpha = ((A_g)_{g\in G},(\alpha_g)_{g\in G})$ be a partial action on a separable \ca\, $A$. Then $\widehat{\alpha}$ has the Rokhlin property if and only if $\alpha$ is global and approximately representable.
\end{cor}

\begin{proof}
   Combine \autoref{dimrep zero global} and \autoref{dual rok and par repr dim}.
\end{proof}

\end{document}